\documentclass[12pt]{article}

\usepackage{setup}
\bibliography{referenceimport}

\title{Braid groups and Burnside groups}
\author{Ethan Dlugie}
\date{July 2026}

\begin{document}

\maketitle

\begin{abstract}
    There exists an exceptional quotient of braid groups $\Br_4 \twoheadrightarrow \Br_3$ that is related to many interesting constructions in algebra, topology, and geometry. This quotient map also descends to a quotient of ``truncated'' braid groups $\Br_4(d) \twoheadrightarrow \Br_3(d)$, which have an added torsion relation on their half twist generators. In this article, we find a presentation for the kernel of this truncated quotient map that takes the form of what we deem a ``primitive'' Burnside group. We give a few finiteness results on these primitive Burnside groups. Our methods are purely group theoretic, but we comment on an interpretation involving Lefschetz fibrations at the end.
\end{abstract}

\section{Introduction}\label{sec: introduction}
The braid groups are a family of groups with wide application in topology, number theory, algebraic geometry, and more. The $n$-strand braid group for us will be defined by the presentation
\begin{equation*}
    \Br_n(d) = \left\langle \sigma_1,\dotsc,\sigma_{n-1} \mid \sigma_i\sigma_{i+1}\sigma_i = \sigma_{i+1}\sigma_i\sigma_{i+1}, \sigma_i\sigma_j=\sigma_j\sigma_i \text{ if } |i-j| > 1 \right\rangle.
\end{equation*}
It is known that the normal subgroup $\llangle \sigma_1\sigma_3^{-1} \rrangle \triangleleft \Br_4$ is isomorphic to a free group of rank 2 (with earliest written sources  \cite[Theorem 7]{Gassner1961OnGroups} and \cite[Theorem 2.6]{Gorin1969AlgebraicBraids}). A fun exercise shows that appending the relation $\sigma_1\sigma_3^{-1}=1$ to $\Br_4$ yields a presentation of $\Br_3$. Therefore we have a short exact sequence
\begin{equation}
    1 \to F_2 \to \Br_4 \xrightarrow{\psi} \Br_3 \to 1 \label{eqn:classis_SES}
\end{equation}
which splits by a standard inclusion of braid groups $\Br_3 \leq \Br_4$.

In this article, we are concerned with what we call \textit{truncated braid groups}. For positive integers $n$ and $d$, such a group is defined as $\Br_n(d) = \Br_n/\llangle \sigma_1^d \rrangle$, where $\llangle - \rrangle$ denotes normal closure. (Note that all $\sigma_i$ are conjugate in $\Br_n$.) These groups were first considered in \cite{Coxeter1959FactorGroup}. The quotient map $\psi$ descends to a quotient map of truncated braid groups:
\begin{equation*}
    \bar\psi:\Br_4(d) \twoheadrightarrow \Br_3(d).
\end{equation*}
For small values of $d$, these groups are finite, allowing for computer investigation. The author was surprised to see that the kernel of $\bar\psi$ in some cases was isomorphic to a free Burnside group of rank 2. Recall that the \emph{free Burnside group} of rank $n$ and exponent $d$ is defined by the presentation
\begin{equation*}
    F_n(d) = \langle x_1,\dotsc,x_n \mid w(x_1,\dotsc,x_n)^d=1 \text{ for all words } w\rangle.
\end{equation*}
The study of the free Burnside groups and their relatives, specifically questions about their finiteness, led to profound developments in algebra in the 20th century (see e.g. \cite{Olshanskii1991GeometryGroups}), including the work of Zelmanov that garnered a Fields medal in 1994.

To identify the kernel of $\bar \psi$ in all cases, we introduce a new construction in the world of Burnside-flavored groups.

\begin{definition}\label{def:primitive_Burnside}
    The \textit{primitive Burnside group} of rank $n$ and exponent $d$ is the group $PF_n(d)$ given by the presentation
    \begin{equation*}
        PF_n(d) = \langle x_1,\dotsc,x_n \mid w(x_1,\dotsc,x_n)^d=1 \text{ for all primitive words } w \rangle.
    \end{equation*}
\end{definition}
Recall here that a word in a free group is \textit{primitive} if it can be extended to a basis for the free group. For example $xy^3$ is primitive in $\langle x,y \rangle$, but $x^2y^2$ is not.

\begin{theorem}\label{thm:kernel_is_primitive_burnside}
    For all $d$, there is a short exact sequence
    \begin{equation*}
        1 \to PF_2(d) \to \Br_4(d) \xrightarrow[]{\bar\psi} \Br_3(d) \to 1
    \end{equation*}
    In fact, the sequence is split, and we have a semidirect product decomposition
    \begin{equation*}
        \Br_4(d) \approx PF_2(d) \rtimes \Br_3(d).
    \end{equation*}
\end{theorem}

Our methods are purely of combinatorial group theory, but they come from topological intuition as explained in \cref{sec:topology_story}.

The truncated braid groups are well understood as examples of complex reflection groups. In particular, their orders are known. This yields the following corollary concerning the orders of some primitive Burnside groups.

\begin{corollary}\label{cor:primitive_Burnside_orders}
    We have the following identifications:
    \begin{itemize}
        \item $PF_2(2) \approx F_2(2)$ of order 4
        \item $PF_2(3) \approx F_2(3)$ of order 27
        \item $PF_2(4)$ and $PF_2(5)$ are both infinite groups
    \end{itemize}
\end{corollary}

Note by contrast that $F_2(4)$ is finite. The group $F_2(5)$ is the lowest case in which the finiteness or infiniteness of a free Burnside group is unknown. Perhaps this knowledge of the associated primitive Burnside group could shed some light on this question.

\subsection*{Acknowledgments}
Thanks to Eduardo Reyes for topological discussions that led to the semidirect product description of \cref{prop:braid_4_semidirect}. Thanks to Seraphina Lee, Faye Jackson, and Trent Lucas for discussions about the Lefschetz fibration interpretation of these results (as discussed in \cref{sec:topology_story}). Google's Gemini assisted drafting \cref{lemma:quotient_of_split_sequence}.

\section{Proofs}\label{sec:proofs}
The main tool for proving \cref{thm:kernel_is_primitive_burnside} is a semidirect product decomposition of $\Br_4$ coming from a section of $\psi$. First, we define the relevant action. Write $F_2 = \langle x,y \rangle$ for the free group of rank 2. Then
\begin{definition}\label{def:modular_action}
    The \textit{modular action} of the 3-strand braid group is the homomorphism $\phi:\Br_3 \to \Aut(F_2)$ given by
    \begin{equation*}
        \phi(\sigma_1) = \begin{cases}
            x \mapsto x\\
            y \mapsto x^{-1}y
        \end{cases}, \quad 
        \phi(\sigma_2) = \begin{cases}
            x \mapsto y\\
            y \mapsto yx^{-1}y
        \end{cases}
    \end{equation*}
\end{definition}
One can readily check that this is a well defined action by automorphisms, i.e. that these maps have inverses and satisfy the braid relation $\phi(\sigma_1\sigma_2\sigma_1)=\phi(\sigma_2\sigma_1\sigma_2)$.

\begin{lemma}\label{lemma:action_on_primitive}
    The braid group acts transitively on conjugacy classes of primitive words in $F_2$ via the modular action.
\end{lemma}
\begin{proof}

    
    Let $u,v \in F_2$ be primitive words. Then there exists some $f \in \Aut(F_2)$ with $f(u)=v$. Considering elements of the group up to conjugacy suggests that we might look at their images in the abelianization $F_2^{ab} \approx \mathbb Z^2$. The group elements map to primitive vectors $[u],[v] \in \mathbb Z^2$, and the induced map $\Aut(F_2) \to \GL(2,\mathbb Z)$ maps the automorphism to a linear transformation $[f] \in \GL(2,\mathbb Z)$. Without loss of generality, e.g. by inverting the value of $f$ on a second element of a basis with $u$, we assume that $[f] \in \operatorname{SL}(2,\mathbb Z)$.

    Now consider the induced map on the automorphisms of the modular action. We see that $\phi(\sigma_1)$ and $\phi(\sigma_2)$ map to the matrices
    \begin{equation*}
        A = \begin{pmatrix}
            1 & -1 \\ 0 & 1 
        \end{pmatrix} \quad \text{and} \quad B = \begin{pmatrix}
            0 & -1 \\ 1 & 2
        \end{pmatrix}
    \end{equation*}
    respectively. A quick computation gives
    \begin{equation*}
        C := A^{-1}BA BA = \begin{pmatrix}
            0 & -1 \\ 1 & 0
        \end{pmatrix},
    \end{equation*}
    so we see that $A$ and $C$ generate all of $\operatorname{SL}(2,\mathbb Z)$. Therefore there is an element $g \in \Br_3$ such that $[\phi(g)] = [f] \in \operatorname{SL}(2,\mathbb Z)$. Nielsen \cite{Nielsen1917DieErzeugenden} showed that the kernel of the map $\Aut(F_2) \to \GL(2,\mathbb Z)$ is exactly the subgroup of inner automorphisms, so $\phi(g)$ and $f$ differ by a conjugation. Therefore, $\phi(g)$ maps the conjugacy class of $u$ to that of $v$.
\end{proof}

Now that we have defined and analyzed the modular action, we can form the associated semidirect product group
\begin{equation*}
    F_2 \rtimes_\phi \Br_3 = \langle x,y,\sigma_1,\sigma_2 \mid \sigma_1\sigma_2\sigma_1=\sigma_2\sigma_1\sigma_2, \sigma_iw\sigma_i^{-1} = \phi(\sigma_i)(w) \rangle_{w \in \{x,y\},i \in \{1,2\}}.
\end{equation*}
This yields a presentation for $\Br_4$.

\begin{proposition}\label{prop:braid_4_semidirect}
    There is an isomorphism $\Br_4 \approx F_2 \rtimes_\phi \Br_3$. 
\end{proposition}
\begin{remark}
    Note that this proposition yields the $d=\infty$ case of \cref{thm:kernel_is_primitive_burnside}, in which case we view the torsion relations as empty.
\end{remark}
\begin{proof}[Proof of \cref{prop:braid_4_semidirect}]
    Define a map $F_2 \rtimes_{\phi} \Br_3 \to \Br_4$ by $x \mapsto \sigma_1^{-1}\sigma_3$, $y \mapsto \sigma_2(\sigma_1^{-1}\sigma_3)\sigma_2^{-1}$, and $\sigma_i \mapsto \sigma_i$. (The last part is the section of $\psi:\Br_4 \to \Br_3$.)
    
    The inverse map is $\Br_4 \to F_2 \rtimes_\phi \Br_3$ via $\sigma_1 \mapsto \sigma_1$, $\sigma_2 \mapsto \sigma_2$, and $\sigma_3 \mapsto \sigma_1 x$.
    That these two maps give well defined homomorphisms is a tedious exercise in applying group relations. Given this, it is clear that the two maps are inverses.
\end{proof}

With this identification, the map $\psi:\Br_4 \to \Br_3$ becomes the projection onto the second factor of the semidirect product. To prove \cref{thm:kernel_is_primitive_burnside}, we must identify the kernel of this projection after quotienting by the normal subgroup generated by $\sigma_1^d$. This follows from the following group theoretic lemma.

\begin{lemma}[Quotients of semidirect products]\label{lemma:quotient_of_split_sequence}
    Suppose that
    \begin{equation*}
        1 \to K \to G \xrightarrow{\psi} Q \to 1
    \end{equation*}
    is a split short exact sequence of groups. Let $N \triangleleft Q$ be a normal subgroup of $Q$. Then the natural quotient map $\bar\psi: G/\llangle N \rrangle \to Q/N$ has kernel isomorphic to $K/[K,N]$ where $[K,N] = \langle knk^{-1}n^{-1} \rangle_{k \in K, n \in N}$ is the mixed commutator subgroup. In addition, the short exact sequence
    \begin{equation*}
        1 \to K/[K,N] \to G/\llangle N \rrangle \xrightarrow{\bar\psi} Q/N \to 1
    \end{equation*}
    is split in the clear way.
\end{lemma}
\begin{proof}
    Write $\psi':G \to Q/N$ for the quotient of the original map $\psi$. The splitting gives a unique way to write elements of $G$ as products $kq$ for some $k \in K$ and $q \in Q$, and then $\psi'(kq)=qN$. Therefore $\ker \psi' = \{kq \mid k \in K, q \in N\} = KN$.

    Since $\psi'(\llangle N \rrangle)$ is trivial in $Q/N$ (by normality of $N$), we have $\llangle N \rrangle \triangleleft KN$. Now $\psi'$ descends to the map $\bar \psi : G/\llangle N \rrangle \to Q/N$ in question. It follows then that $\ker(\bar \psi) = KN/\llangle N \rrangle$.

    Elements of $\llangle N \rrangle$ are products of elements of the form
    \begin{equation*}
        (kq)n(kq)^{-1} = k(qnq^{-1})k^{-1}(qnq^{-1})^{-1}(qnq^{-1}) \in [K,N]N.
    \end{equation*}    
    It also clear that $[K,N]N \leq \llangle N \rrangle$, so that $\llangle N \rrangle = [K,N]N$. So now we have
    \begin{equation*}
        \ker(\bar \psi) = KN/([K,N]N).
    \end{equation*}

    Since $N$ normalizes $K$, it follows that $[K,N] \leq K$. Furthermore, this subgroup is normal in $K$ by the identity
    \begin{equation*}
        \bar k [k,n]\bar k^{-1} = [\bar k k,n][\bar k,n]^{-1} \in [K,N].
    \end{equation*}
    So we finally define a map $KN \to K/[K,N]$ by $kn \mapsto k[K,N]$. This is a homomorphism because
     \begin{align*}
         (k_1n_1)(k_2n_2) = k_1(n_1&k_2n_1^{-1})n_1n_2\\ & \mapsto k_1n_1k_2n_1^{-1}[K,N] = k_1k_2[k_2^{-1},n_1][K,N] = k_1k_2[K,N].
     \end{align*}
     The homomorphism is clearly surjective, and its kernel is $\{kn \mid k \in [K,N]\} = [K,N]N$. The first isomorphism theorem finishes the proof.
\end{proof}

\begin{proof}[Proof of \cref{thm:kernel_is_primitive_burnside}]
    \cref{prop:braid_4_semidirect} allows us to apply \cref{lemma:quotient_of_split_sequence} with $K=F_2$, $G=\Br_4$, $Q=\Br_3$ and $N = \langle g\sigma_1^dg^{-1} \mid g \in \Br_3 \rangle$. For convenience, we will write $g .w = \phi(g)(w)$ for $g \in \Br_3$ and $w \in F_2$. Then we obtain the kernel of $\psi:\Br_4(d) \to \Br_3(d)$ as
    \begin{align*}
        F_2/\langle wn^dw^{-1}n^{-d} \rangle_{w \in F_2, n \in N} & = \langle x,y \mid w = (g\sigma_1^dg^{-1})w(g\sigma_1^dg^{-1})^{-1} \rangle_{w \in F_2, g \in \Br_3}\\
        & = \langle x,y \mid w = (g\sigma_1^dg^{-1}).w \rangle_{w \in F_2, g \in \Br_3}.
    \end{align*}
    Note here that this conjugation relation need only be recorded for each $g \sigma_1^d g^{-1}$; the relation follows in turn for products of such elements of $N$.
    
    Note also that we need not include such relations for each $w \in F_2$
    . If for each conjugate in $\Br_3$, we record the relation for a pair of words in a suitable basis for $F_2$, then the relation follows for all elements of the free group. Concretely, for each $g \in \Br_3$ the pair $\{g.x,g.y\}$ forms a basis for $F_2$. So we have further that the kernel is presented as
    \begin{align*}
        F_2/\langle wn^dw^{-1}n^{-d} \rangle_{w \in F_2, n \in N} & = \langle x,y \mid g.x = (g\sigma_1^dg^{-1}).(g.x), g.y=(g\sigma_1^dg^{-1}).(g.y) \rangle_{g \in \Br_3}\\
        \qquad 
        & = \langle x,y \mid g.x=g.x, g.y=(g.x)^{-d}(g.y) \rangle_{g \in \Br_3}\\
        &= \langle x,y \mid 1=(g.x)^d \rangle_{g \in \Br_3}.
    \end{align*}
    Finally by \cref{lemma:action_on_primitive}, the orbit $\{g.x\}_{g \in \Br_3}$ covers exactly the set of primitive conjugacy classes in $F_2$. This yields the primitive Burnside presentation.
\end{proof}

\begin{proof}[Proof of \cref{cor:primitive_Burnside_orders}]
    We clearly have quotient maps from the primitive Burnside groups to the free Burnside groups: $PF_n(d) \twoheadrightarrow F_n(d)$. The first few instances of free Burnside groups already have finite presentations with power relations on primitive elements:
    \begin{align*}
        F_2(2) &= \langle x,y \mid x^2=y^2=(xy)^2=1 \rangle\\
        F_2(3) &= \langle x,y \mid x^3=y^3=(xy)^3=(xy^{-1})^3=(x^{-1}y)^3=1 \rangle
    \end{align*}
    Thus we also have quotient maps $F_2(d) \twoheadrightarrow PF_2(d)$ for $d=2,3$. The isomorphisms follow.

    \begin{figure}
        \centering

        \begin{overpic}[
            page=1, scale=0.8,
            viewport={0.3in 8.5in 4.0in 10.7in},
            clip, grid=false
        ]{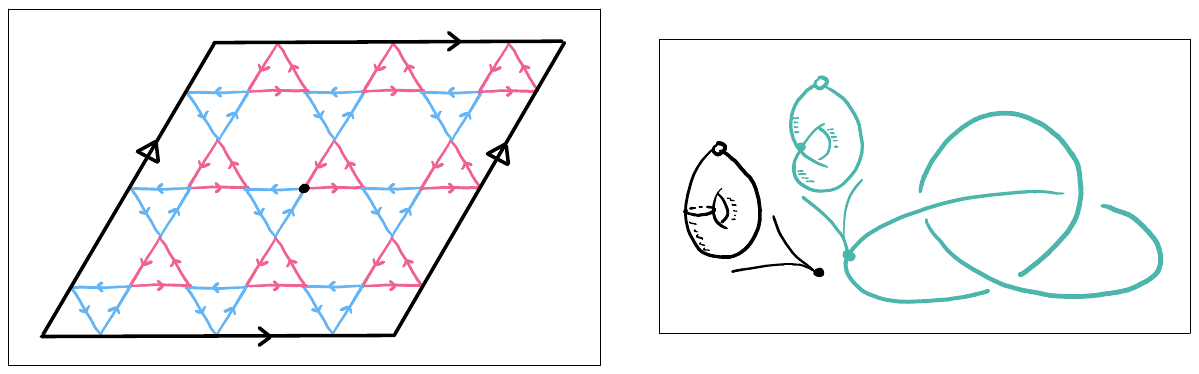}
        \end{overpic}
        \caption[Cayley graph of free Burnside group of rank 2 and exponent 3 in the torus]{A Cayley graph for a group all of whose elements have order 3, embedded in a torus $T^2$. This group is a quotient of $F_2(3)$, and it has order $27=\#F_2(3)$, so it must be isomoprhic to $F_2(3)$. The relations can be read off from the cells in the obvious 2-complex and from generators for $\pi_1(T^2)$.}
        \label{fig:Burnside_2_3}
    \end{figure}
    
    For the statements about the orders of the groups, we turn to Coxeter's results about truncated braid groups \cite{Coxeter1959FactorGroup}. Coxeter showed in particular that
    \begin{align*}
        \#\Br_3(2) &= 6 & \#\Br_4(2) &= 24\\
        \#\Br_3(3) &= 24 & \#\Br_4(3) &= 648\\
        \#\Br_3(4) &< \infty & \#\Br_4(4) &= \infty\\
        \#\Br_3(5) &< \infty & \#\Br_4(5) &= \infty\\
    \end{align*}
    Then the index of $PF_2(d)$ in $\Br_4(d)$ is given by $\#\Br_3(d)$.
\end{proof}

\section{A topological interpretation}\label{sec:topology_story}
The ideas in this section are shared without proof, but we think they are useful to share nonetheless. The map $\psi:\Br_4 \to \Br_3$ is not purely group theoretic. It is the map on $\pi_1$ induced by a map of configuration spaces $\operatorname{Conf}_4(\mathbb C) \to \operatorname{Conf}_3(\mathbb C)$ given as $\{a,b,c,d\} \mapsto \{ab+cd,ac+bd,ad+bc\}$. See \cite{Chen2025FamiliesStrands} for more discussion of this map and its relatives. This map of configuration spaces turns out to be a fiber bundle with fiber the once-punctured torus $S_{1,1}$. The long exact sequence on homotopy groups yields the short exact sequence \eqref{eqn:classis_SES}.

To obtain the analogous map on truncated braid groups, one must perform a sort of orbifold filling on the configuration spaces. For $\operatorname{Conf}_3(\mathbb C)$, which is homotopy equivalent to the trefoil knot complement in the 3-sphere, this is rather direct. See e.g. my earlier article \cite{Dlugie2025TheTorsion}. Call the filled space $\operatorname{Conf}_3[d]$.

For $\operatorname{Conf}_4(\mathbb C)$, one has to be a bit careful. The obvious compactification of configuration space needs to be blown up along the part corresponding to triple collisions of points, along the lines of \cite{Fulton1994ASpaces}. After doing this, orbifold surgery should be more direct, yielding a space $\operatorname{Conf}_4[d]$. The map on orbifolds is no longer a fiber bundle but should be thought of as an \textit{orbifold Lefschetz fibration}, since the fibers over the singular part of $\operatorname{Conf}_3[d]$ are pinched surfaces and have nontrivial ambient orbifold isotropy. The cartoon is something like \cref{fig:orbifold_Lefschetz}.

\begin{figure}
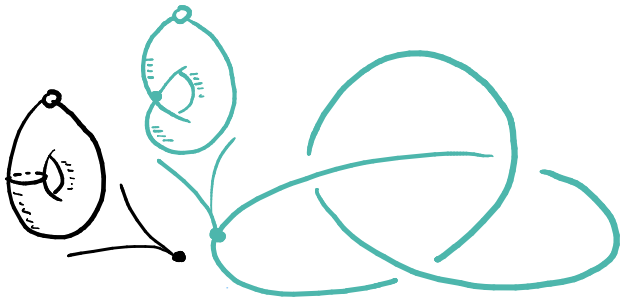

        \centering

        \begin{overpic}[
            page=1,
            viewport={4.6in 8.7in 8.0in 10.5in},
            clip, grid=false
        ]{figures.pdf}
        \end{overpic}
        \caption[Cartoon of the orbifold Lefschetz fibration]{The space $\operatorname{Conf}_3[d]$ is the 3-sphere with a $\mathbb Z/d$ orbifold locus along a trefoil knot. The space $\operatorname{Conf}_4[d]$ is a $S_{1,1}$ bundle on this space away from the knot, while the fibers collapse to pinched surfaces along the knot. The pinch points correspond to 4-point configurations with two double collisions.}
        \label{fig:orbifold_Lefschetz}
    \end{figure}

The typical short exact sequence of Lefschetz fibrations should have an orbifold version akin to
\begin{center}
    \begin{tikzcd}[row sep=small]
        1 \arrow[r] & {\pi_1(S_{1,1})/C^d} \arrow[r] & {\pi_1(\operatorname{Conf}_4[d])} \arrow[r] & {\pi_1(\operatorname{Conf}_3[d])} \arrow[r] & 1 \\
                    &                              & \Br_4(d) \arrow[u, "\approx"]     & \Br_3(d) \arrow[u, "\approx"]     &  
    \end{tikzcd}
\end{center}
where $C \subset \pi_1(S_{1,1})$ denotes the set of so-called vanishing cycles, a particular set of simple closed curves in the fiber along with their orbits under the monodromy action. In this case, the monodromy action $\Br_3 \curvearrowright \pi_1(S_{1,1})$ is given by \cref{def:modular_action}, and $C$ consists of all simple closed curves in $S_{1,1}$ (which are the same as primitive elements of $F_2$).

This topological interpretation led me to discover the main result of this work. I find the topological story more enlightening, but I felt that the group theory was more direct than trying to develop a theory of orbifold Lefschetz fibrations. Perhaps the development of such a theory could come in future work.
\printbibliography

\end{document}